\documentclass[a4paper]{article}
\usepackage{amsfonts,amsthm,latexsym,amsmath,mathrsfs,amscd,geometry}
\usepackage{amssymb}
\usepackage{dsfont}
\usepackage{latexsym}
\usepackage{appendix}
\usepackage{enumerate}
\usepackage{color}
\usepackage{upgreek}
\usepackage{geometry}
\numberwithin{equation}{section} 
\newtheorem{theorem}{Theorem}[section]
\newtheorem{lemma}[theorem]{Lemma}
\newtheorem{proposition}[theorem]{Proposition}

\newtheorem{remark}[theorem]{Remark} 
\makeatletter

\newcommand{\Rmnum}[1]{\expandafter\@slowromancap\romannumeral #1@}
\makeatother

\allowdisplaybreaks

\usepackage[colorlinks,linkcolor=blue]{hyperref} 

\begin{document}
\title{Optimal time-decay estimates for a diffusive Oldroyd-B model}
\author{Yinghui Wang\footnote{School of Mathematics, South China University of Technology, Guangzhou 510641, China. E-mail: yhwangmath@163.com} 
}
\date{}
\maketitle
\begin{abstract}
In this paper, we study the optimal time decay rates for the higher order  spatial derivatives of  solutions to a diffusive Oldroyd-B model.  
As pointed out in the Section 1.2 of Huang-Wang-Wen-Zi (J. Differential Equations 306: 456–491, 2022), how to estiblish the optimal decay estimate for the highest-order spatial derivatives of the solution  to this model is still an open problem. Motivated by Wang-Wen (Sci. China Math. 65: 1199–1228, 2022), we give a positive answer to this problem via some delicate analyses on the low and high frequency parts of the solution.
\end{abstract}
\bigbreak \textbf{{\bf Key Words}:}   Oldroyd-B model; Optimal decay rate;  Zero viscosity
\bigbreak  {\textbf{AMS Subject Classification 2020:} 35Q35, 76A10, 74H40.

\section{Introduction}
In this paper, we study the following diffusive Oldroyd-B system
\begin{eqnarray} \label{system}
    \begin{cases}
         \partial_tu+u\cdot\nabla u+\nabla p-\epsilon\Delta u=\kappa {\rm div}\tau,\\
         \partial_t\tau+u\cdot\nabla\tau-\mu\Delta\tau+\beta\tau=Q(\nabla u,\tau)+\alpha\mathbb{D}u,\\
         {\rm div}u=0,\\
         (u,\tau)(x,0)=(u_0,\tau_0),
    \end{cases}
    \end{eqnarray} on $\mathbb{R}^3\times (0,\infty)$, which is used to describe the motion of viscoelastic fluids. Here, $u=(u_1,u_2,u_3)^\top:\mathbb{R}^3\rightarrow\mathbb{R}^3$ is the velocity field of fluid, symmetric matrix $\tau\in \mathbb{S}_3(\mathbb{R})$ is the tangential part and non-Newtonian part of the stress tensor, $p\in\mathbb{R}$ is the pressure function of the fluid, $\mathbb{D}u=\frac12\left(\nabla u+\left(\nabla u\right)^\top\right)$ is the deformation tensor, and   
    $$Q(\nabla u,\tau)=\Omega\tau-\tau\Omega+b(\mathbb{D}u \tau+\tau\mathbb{D}u)$$
     admits the usual bilinear form  with the skew-symmetric part of velocity gradient $\Omega=\frac12\left(\nabla u-\left(\nabla u\right)^\top\right)$ and constant  $b\in [-1,1]$.
    The parameters   $\kappa$,   $\beta$ and $\alpha$ satisfy that $  \kappa,\beta,\alpha>0$. Moreover,
     $\epsilon \geq 0$ is the viscosity  coefficient of the fluid, and $\mu\geq0$ is the center-of-mass diffusion coefficient.  The system \eqref{system} was first proposed by Oldroyd in 1958 (\cite{Oldroyd_1958}).

     As pointed out by Bhave, Armstrong and Brown (\cite{Bhave_Armstrong_Brown_1991}), the diffusion coefficient $\mu$ is significantly smaller than other effects, the diffusive term $\mu\Delta \tau$ is ignored in the classical Oldroyd-B model (the non-diffusive  model). However, in the recent work by J. M\'{a}lek, V. Prů\v{s}a, T. Sk\v{r}ivan and E. S\"{u}li (\cite{Malek_etal_2018}), the authors showed that the stress
     diffusion term can be interpreted either as a consequence of a nonlocal energy storage mechanism or as a consequence of a nonlocal entropy production mechanism.

     The non-diffusive Oldroyd-B model (i.e.,  \eqref{system} with $\mu=0$) has been well studied by many authors in  the field of mathematics and physics. For the existence results of strong solution, one can refer to the early works by  Guillop\'e and Saut (\cite{Guillop-1}),   Fern\'andez-Cara, Guillen and Ortega (\cite{Fernandez-Guillpen-Ortega}), and by Molinet-Talhouk (\cite{Molinet-Talhouk}), and the recent works by Hieber, Naito and Shibata (\cite{Hieber-Naito-Shibata}),  Fang, Hieber and Zi (\cite{Fang-Hieber-Zi}), and by Zi, Fang and Zhang (\cite{Fang-Zhang-Zi}).
     Lions and Masmoudi (\cite{Lions_Masmoudi_2000}) obtained a global existence result for weak solutions of the corotational  model ($b=0$). For the results about blow-up criteria, one can refer to the works by  Chemin, and Masmoudi (\cite{Chemin_Masmoudi_2001}), Lei, Masmoudi, and Zhou (\cite{Lei_Masmoudi_Zhou_2010}), and by Sun and Zhang (\cite{Sun_Zhang_2011}). The large time behaviors of the solutions are investigated by Hieber, Wen and Zi (\cite{Hieber_Wen_Zi_2019}) and by Huang, Wang, Wen and Zi (\cite{HWWZ_2022}).
     For a detailed review of those works, one can refer for instance to the Introduction in \cite{HWWZ_2022},  the review papers by Lin (\cite{Lin_2012}) and by Renardy and Thomases  (\cite{Renardy_Thomases_2021}), and the references therein. 
     
     However, there are only a few works on the mathematical theory of the diffusive Oldroyd-B model (i.e., (\eqref{system} with $\mu>0$). In the case $\epsilon,\mu>0,$ Constantin and Kliegl (\cite{Constantin_Kliegl_2012}) obtained the global existence and uniqueness of strong solutions in $\mathbb{R}^2.$ In the case $\epsilon = 0,\mu>0,$ Elgindi and Rousset (\cite{Elgindi_Rousset_2015}) obtained a global wellposedness result provided that the initial data are small in $H^s(\mathbb{R}^2)(s>2)$. This result was extended to 3-D case with small initial data in $\mathbb{R}^3$ by Elgindi and Liu (\cite{Elgindi_Liu_2015}).  Recently, Huang, Wang, Wen and Zi (\cite{HWWZ_2022}) proved the global wellposedness to the Cauchy problem in 3-D and deduced some time decay estimates. Liu, Wang and Wen (\cite{Liu_Wang_Wen_2022}) proved a similar results for the  compressible counterpart of this model.  Here, we restate the result of \cite{HWWZ_2022} in the following Proposition. 
    \begin{proposition}[\cite{HWWZ_2022}]  \label{proposition_HWWZ} Assume that $(u_0,\tau_0)\in H^3(\mathbb{R}^3)$. For any given $\epsilon\geq 0$ and $\mu>0$, there exists a sufficiently small constant 
        $0<\varepsilon_0\leq  1 $  satisfying  \eqref{eq_varepsilon}  such that the Cauchy problem (\ref{system}) admits a unique global solution $(u,\tau)\in C([0,+\infty);H^3(\mathbb{R}^3))$ satisfying
        \begin{equation}\label{eq_smallness} 
            \setlength{\abovedisplayskip}{3pt} 
            \setlength{\belowdisplayskip}{3pt}
              { \|u(t)\|_{H^3}^2 + \|\tau(t)\|_{H^3}^2}
              +\int_0^t\Big(\epsilon\|\nabla u (s)\|_{H^3}^2+\|\nabla u (s)\|_{H^2}^2 + \mu\|\nabla\tau (s)\|_{H^3}^2+\|\tau (s)\|_{H^3}^2\Big)ds\leq C_1\varepsilon_0^2,  
        \end{equation} for $t\geq0$, provided that $ 
            \|u_0\|_{H^3(\mathbb{R}^3)}+\|\tau_0\|_{H^3(\mathbb{R}^3)}\leq \varepsilon_0,$
          where $\varepsilon_0$ is a constant   independent of $\epsilon$ and $t$, and the positive constant $C_1$ may depend on $\mu$   but independent of $\epsilon$  and $t$.

          Assume further that $(u_0,\tau_0)\in L^1(\mathbb{R}^3)$. Then the following upper time-decay estimates   hold:
          \begin{equation}\label{eq_1} 
             \|\nabla^ku (t)\|_{L^2}\leq C_2(1+t)^{-\frac34-\frac{k}{2}},\, 
             \|\nabla^{k_1}\tau (t)\|_{L^2}\leq C_2(1+t)^{-\frac54-\frac{k_1}{2}},
          \end{equation}
              and 
             \begin{equation} \label{eq_2} 
             \,\|\nabla^3 u   (t) \|_{L^2} +  \|\nabla^2 \tau  (t) \|_{H^1} \leq C_2 (1+t)^{-\frac{7}{4}},
            \end{equation}
            for any $t\geq0$, where
            $k=0,1,2,\,k_1=0,1,$ and the constant $C_2$ depends only on $\|(u_0,\tau_0)\|_{H^3\cap L^1}$ and $C_1$. 
    \end{proposition} 
    \begin{remark}
        For readers' convenience,  we remark that estimates in \eqref{eq_1} are partial conclusions of Theorem 1.2 in \cite{HWWZ_2022} (see, \cite{HWWZ_2022}, Theorem 1.2, part (i)). And \eqref{eq_2} is a byproduct of the proof
        of \eqref{eq_1} which was stated in  Lemma 4.5  of \cite{HWWZ_2022} (see (4.15) in page 478 of \cite{HWWZ_2022}).
    \end{remark} 
    \begin{remark}
       For a proof of Proposition \ref{proposition_HWWZ}, see Theorems 1.1 and 1.2 in \cite{HWWZ_2022}. The assumption \eqref{eq_varepsilon} for $\varepsilon_0$ is also assumed in \cite{HWWZ_2022} via a qualitative  statement   ``$\varepsilon_0$ is sufficiently small constant''. Here, we give the quantitative assumption  \eqref{eq_varepsilon} is  in order to clarify the proof, see the proof of Lemma \ref{lemma_decay_u_3} for details. 
    \end{remark} 
    \begin{remark}
        Besides the the results stated in Proposition \ref{proposition_HWWZ},   the authors also proved a similar result for the non-diffusive model in \cite{HWWZ_2022}. From the proofs in \cite{HWWZ_2022}, one can find that the non-diffusive model ($\epsilon>0,\mu\geq 0$) is much easier to handle than the  diffusive model ($\epsilon\geq 0,\mu > 0$). Therefore, in the present paper, we only deal with the diffusive model.
    \end{remark}
    \begin{remark}
         In the previous works (e.g. \cite{HWWZ_2022},\cite{Liu_Wang_Wen_2022} and the references therein), the optimal time decay estimate for the highest-order spatial
     derivatives of $\tau$ is not obtained for both \eqref{system} and the compressible system.
     \end{remark}
\subsection{Main result}
The main result of the present paper is to investigate the 
 optimal decay estimate for the
highest-order spatial derivatives of the solution obtained in Proposition \ref{proposition_HWWZ}. 
\begin{theorem}\label{thm_more_decay}Under the assumptions of Proposition \ref{proposition_HWWZ}, 
     for any given $\epsilon\geq 0$ and $\mu>0$,  the solution $(u,\tau)$ of problem \eqref{system} enjoys the following further optimal decay estimates 
\begin{eqnarray*}
     \|\nabla^3 u (t)\|_{L^2} + \|\nabla^2 \tau (t)\|_{L^2}\leq C_3(1+t)^{-\frac94},\,\,\, \|\nabla^3 \tau (t)\|_{L^2} \leq C_3(1+t)^{-\frac{11}4}.
\end{eqnarray*}
\end{theorem}
\begin{remark}
    The conclusion of Theorem \ref{thm_more_decay} is optimal in the sense that the decay estimates for $(u,\tau)$ are the same as the decay estimates for  solution of the linearized system \eqref{u-sigma_linear}. This result gives an positive answer to the open problem asked in the Section 1.2 of  \cite{HWWZ_2022}, refer to page 461 in \cite{HWWZ_2022} for details. 
\end{remark} 
The main ideas of this work are as follows. Motivated by 
Wang and Wen (\cite{Wang_Wen_2022}), we make some delicate estimates to remove the low-frequency part of the energy functional to obtain the optimal decay estimate for $\|\nabla^3 u\|_{L^2},$ see Lemma \ref{lemma_decay_u_3} for details. The main difficulties arise in the decay estimate of $\|\nabla^3 \tau\|_{L^2}$ due  to the fact that the decay rate of $\|\nabla^3 \tau\|_{L^2}$ should be fast than that of $\|\nabla^3 u\|_{L^2}$ as indicated by decay properties of the linearized problem, see Sections 1.2 and 1.3 of \cite{HWWZ_2022} for a detailed illustration. Our 
observation here is that the high-frequency part $\|\nabla^3 (u^h,\tau^h)\|_{L^2}$ decays faster than $\|\nabla^3(u,\tau)\|_{L^2}$. Then, we make  full use of the low-high-frequency decomposition technique to prove our main result, see Lemma \ref{lemma_tau_decay} for details.

The rest of the paper is organized as follows. In Section \ref{section_pre}, we recall some known  results  which will be used in the proof of the main Theorem. In Section \ref{section_proof}, we improve the decay rates of the solution successively and prove Theorem \ref{thm_more_decay}.
\section{Preliminaries}\label{section_pre}
\subsection*{Notations}
In this paper, some standard notations are used. We use $C$ to denote the generic positive constant which may depend on the initial value and some other known constants but independent of time $t$ and the variable parameter $\epsilon$.
We also use $B_i,C_i (i=0,1,2,\cdots)$  to denote the specific constants which are necessary to clarify the proofs. 
 Some other notations are stated as below:

  ``$G\lesssim F$'' means that there exists a positive constant $C$ such that ``$G\leq C F$''. 
      We simply use the notations $\|\cdot\|_{L^p}$ and $ \|\cdot\|_{H^k}$ ($1\leq p \leq \infty,k=1,2\cdots$ )  for the norm of  spaces $L^p(\mathbb{R}^3)$ and $H^k(\mathbb{R}^3).$ {  As usual,} $\langle\cdot,\cdot\rangle$ is the inner-product in $L^2(\mathbb{R}^3)$. 
      For a Banach space $X$, we  write $\|f\|^2_{X} + \|g\|^2_{X}$ as $\|(f,g)\|^2_{X}$. 

      For $f\in L^2(\mathbb{R}^3)$, we denote the Fourier transform of $f$ by $$\hat{f}(\xi):= \mathscr{F}[f](\xi):=(2\pi)^{-\frac{3}{2}}\int_{\mathbb{R}^3}e^{-ix\cdot\xi}f(x)\mathrm{d }x.$$
     Let $\Lambda := \sqrt{-\Delta}$ and $\mathbb{P}$ be the  Leray
    projector which can be  represented via the Fourier transform  as follows
    $$ \Lambda f = \mathscr{F}^{-1}[|\xi|\hat{f}], ~\text{ and }
    (\mathbb{P} u)^j = \mathscr{F}^{-1}\left[\left(\delta_{j,k}-\frac{\xi_j\xi_k}{|\xi|^2}\right) (\hat{u}^k)\right].$$ 
       Let $\phi(\xi)$ be a smooth cut-off function satisfying that 
    \begin{equation*}
        0\leq \phi(x)\leq 1, ~\text{ and }
        \begin{cases}
            \phi(\xi) =1, &\text{ for }|\xi|\leq \frac{R}{2},\\ 
            \phi(\xi) =0, &\text{ for }|\xi|\geq  R,
        \end{cases}
    \end{equation*}  where $R >0$ is a constant defined in Lemma \ref{theorem_Green}.
    Set $$\phi_0(\xi):=\phi(\xi),~\phi_1(\xi):=1-\phi(\xi),~
\tilde{\phi}_0(\xi):=1-\left(1-\phi(\xi)\right)^2,~\tilde{\phi}_1(\xi):={  \phi^2_1(\xi) = }\left(1-\phi(\xi)\right)^2.$$ 
    {  Using the above notations,} for  $f\in L^2(\mathbb{R}^3),$ we can define the  low frequency and high frequency decomposition as follows
    \begin{equation} \label{eq_decom}
        f(x) = f^\ell(x) + f^{h}(x) = f^{\tilde{\ell}}(x) + f^{\tilde{h}}(x),
    \end{equation}
    where 
    \begin{equation*} 
        f^\ell(x) := \phi_0(D)f(x),~f^{h}(x):= \phi_1(D)f(x),~
         f^{\tilde{\ell}}(x):= \tilde{\phi}_0(D)f(x),~f^{\tilde{h}}(x):=\tilde{\phi}_1(D)f (x),
    \end{equation*}
    with the convention that $\phi(D)f(x):= \mathscr{F}^{-1}[\phi\hat{f}](x)$ for smooth $\phi$. Then, we use the followings notations 
    $$ \Lambda^k\phi(D)f(x) = \mathscr{F}^{-1}[|\xi|^k\phi\hat{f}](x).$$
      As a consequence of { Plancherel's} theorem, for $f\in H^3(\mathbb{R}^3)$, we have that 
    \begin{eqnarray*}
        \|\nabla^k f^{ \ell }\|_{L^2}+ \|\nabla^k f^{ h }\|_{L^2} + \|\nabla^k f^{\tilde{\ell}}\|_{L^2}+ \|\nabla^k f^{\tilde{h}}\|_{L^2} \lesssim  \|\nabla^k f \|_{L^2} ,\,\,\text{ for }k=0,1,2,3,
    \end{eqnarray*}
    and
\begin{eqnarray*}
    \|\nabla^k f^{\tilde{h}}\|_{L^2}\lesssim \|\nabla^k f^{h}\|_{L^2} \lesssim \|\nabla^{k+1} f^{h}\|_{L^2} \lesssim \|\nabla^{k+1} f\|_{L^2},\,\,\text{ for }k=0,1,2.
\end{eqnarray*} 
In order to supplement the dissipation of $u,$ similar to the treatment in \cite{HWWZ_2022}, applying the
Leray projector $\mathbb{P}$ and  the operator $\Lambda^{-1}\mathbb{P}{\rm div}$ to $\eqref{system}_1$  and $\eqref{system}_2$, respectively,
and denoting by 
 $$\sigma := \Lambda^{-1}\mathbb{P}{\rm div}\tau  
\text{ with } \left(\hat{\sigma}\right)^j=i\left(\delta_{j,k}-\frac{\xi_j\xi_k}{|\xi|^2}\right)\frac{\xi_l}{|\xi|}\left(\hat{\tau}\right)^{l,k},$$ 
we obtain
\begin{equation} \label{u-sigma}
    \setlength{\abovedisplayskip}{3pt} 
            \setlength{\belowdisplayskip}{3pt}
    \begin{cases}
         \partial_tu-\epsilon\Delta u-\kappa \Lambda\sigma=\mathcal{F}_1,\\
         \partial_t\sigma-\mu\Delta\sigma+\beta\sigma+\frac{\alpha}{2}\Lambda u=\mathcal{F}_2,
    \end{cases}
    \end{equation} where the nonlinear terms are stated as below:
    \begin{equation*} 
           \mathcal{F}_1=-\mathbb{P}\left(u\cdot\nabla u\right),\,\, 
          \mathcal{F}_2=-\Lambda^{-1}\mathbb{P}{\rm div}\left(u\cdot\nabla\tau\right)+\Lambda^{-1}\mathbb{P}{\rm div}Q(\nabla u,\tau). 
    \end{equation*}
    Next, we consider   the linearized system of \eqref{u-sigma}:
    \begin{eqnarray} \label{u-sigma_linear}
    \begin{cases}
    \partial_tu-\epsilon\Delta u-\kappa \Lambda\sigma=0,\\
    \partial_t\sigma-\mu\Delta\sigma+\beta\sigma+\frac{\alpha}{2}\Lambda u=0,\\
    (u,\sigma)(x,0)=(u_0,\sigma_0)(x).
    \end{cases}
    \end{eqnarray}
    For system \eqref{u-sigma_linear}, we have the following Lemmas \ref{u_sigma_Fourier}, \ref{theorem_Green} and \ref{low_feq_est} which are proved in \cite{HWWZ_2022}.
    \begin{lemma}[Lemma  2.1  in \cite{HWWZ_2022}]\label{u_sigma_Fourier}
        Fourier transform of the solution to  system (\ref{u-sigma_linear}) can be solved by
        \begin{align*}
            \hat{u}^{j} =&{}\big(\mathcal{G}_3(\xi,t)-\epsilon|\xi|^2\mathcal{G}_1(\xi,t)\big)\hat{u}^{j}_0 
            +\kappa|\xi|\mathcal{G}_1(\xi,t)\hat{\sigma}^j_0,\\ 
            \hat{\sigma}^j=&{}-\frac{\alpha}{2}|\xi|\mathcal{G}_1(\xi,t)\hat{u}^{j}_0
            +\big(\mathcal{G}_2(t,\xi)+\epsilon|\xi|^2\mathcal{G}_1(\xi,t)\big)\hat{\sigma}^j_0,
        \end{align*} 
        for $j= 1,2,3,$ and 
        \begin{eqnarray*}\label{G}
           \mathcal{G}_1(\xi, t)=\frac{e^{\lambda_+t}-e^{\lambda_-t}}{\lambda_+-\lambda_-}, \ \mathcal{G}_2(\xi, t)=\frac{\lambda_+e^{\lambda_+t}-\lambda_-e^{\lambda_-t}}{\lambda_+-\lambda_-}, \ \mathcal{G}_3(\xi, t)=\frac{\lambda_+e^{\lambda_-t}-\lambda_-e^{\lambda_+t}}{\lambda_+-\lambda_-}.
        \end{eqnarray*}  
        \end{lemma}
        \begin{lemma}[Proposition  2.3  and Lemma  4.5  in \cite{HWWZ_2022}] \label{theorem_Green}  There exist positive constants $R=R(\alpha,\kappa,\beta)$, $\theta=\theta(\alpha,\kappa,\beta)$ and $K=K(\alpha,\kappa,\beta)$ such that
            \begin{eqnarray*} 
                 \left|\mathcal{G}_1(\xi, t)\right| + \left|\mathcal{G}_3(\xi,t)\right| \leq Ke^{-\theta|\xi|^2t},\,\,\,\,
                 |\mathcal{G}_2(\xi,t)|
            \leq K\left(|\xi|^2 e^{-\theta|\xi|^2t} + e^{-\frac{\beta t}{2}}\right), \,\,\text{ for any }|\xi|\leq R \text{ and }t>0.
            \end{eqnarray*} 
            \end{lemma}
            \begin{lemma}[Lemma  4.1  in \cite{HWWZ_2022}]\label{low_feq_est}
                Letting $(u,\sigma)$ be the solution of problem \eqref{u-sigma}, then we have the following time-decay estimates, for the low-frequency part of $u$, 
                \begin{align*}
                    \left(
                    \int_{|\xi|\leq R}
                     |\xi|^{2k}|\hat{u}|^2
                      {\rm d}\xi
                    \right)^\frac{1}{2}
                     \lesssim &\,
                     (1+t)^{-\frac34-\frac{k}{2}}
                       \|(u_0,\tau_0)\|_{L^1} 
                       +
                      \int_0^{\frac{t}{2}}
                      (1+t-s)^{-\frac34-\frac{k}{2}}
                      \left\|\left(\mathcal{F}_1,\mathcal{F}_2\right)^\top(s)\right\|_{L^1}
                      {\rm d}s\\ 
                       & 
                      +
                      \int_{\frac{t}{2}}^t
                      (1+t-s)^{-\frac{k}{2}}
                      \left\|\left(\mathcal{F}_1,\mathcal{F}_2\right)^\top(s)\right\|_{L^2}\mathrm{d}s. 
                \end{align*} 
            \end{lemma}  
            Next, we recall some  useful inequalities in the following Lemmas.
            \begin{lemma}[\cite{Adams_Fournier_2003,Taylor_2011_3}]\label{lemma_Sobolev}
                Let $f\in H^2(\mathbb{R}^3)$. Then, we have 
                \begin{align*}
                    \|f\|_{L^q} \lesssim&\,   \|f\|_{H^1},~2\leq q\leq 6, \\ 
                    \|f\|_{L^6} \lesssim&\,   \|\nabla f\|_{L^2},\\ 
                    \|f\|_{L^\infty} \lesssim&\,  \|\nabla f\|_{L^2}^{\frac{1}{2}} \|\nabla^2 f\|_{L^2}^{\frac{1}{2}} \lesssim    \|\nabla f\|_{H^1}.
                \end{align*}
            \end{lemma}
            \begin{lemma}[\cite{Majda_Bertozzi_2002}]\label{lemma_commutator}Let $k\geq 1$ be an integer and $f,g \in L^\infty(\mathbb{R}^3)\cap H^k({\mathbb{R}^3})$, it holds that
                 \begin{equation*}
                    \|\nabla^k(fg)\|_{L^2} \lesssim \|f\|_{L^\infty}\|\nabla^k g\|_{L^2} + \|\nabla^k f\|_{L^2}\|  g\|_{L^\infty},
                 \end{equation*}
                 and 
                 \begin{equation*}
                    \|\nabla^k(fg) - f\nabla^k g\|_{L^2}
                    \lesssim
                    \|\nabla f\|_{L^\infty}\|\nabla^{k-1} g\|_{L^2} + \|\nabla^k f\|_{L^2}\|  g\|_{L^\infty}.
                 \end{equation*}
            \end{lemma} 
            \section{Proof of Theorem \ref{thm_more_decay}}\label{section_proof} 
            Noting that 
            \begin{equation*}
                \setlength{\abovedisplayskip}{1pt} 
            \setlength{\belowdisplayskip}{3pt}
                \langle |\Lambda^3 u^h,\Lambda^2\sigma^h \rangle| \leq
                \|\Lambda^3 u^h\|_{L^2}\|\Lambda^2\sigma^h\|_{L^2}\leq  B_0 \Big(\|\Lambda^3 u^h\|_{L^2}^2 + \|\Lambda^2\sigma^h\|_{L^2}^2\Big),
            \end{equation*}
            define the temporal energy functional that 
            \begin{eqnarray*}
                \mathcal{H}_3(t):= \alpha\|\nabla^3u \|_{L^2}^2+\kappa\|\nabla^3\tau\|_{L^2}^2,~~\tilde{\mathcal{H}}_3(t):=\alpha\|\nabla^3u\|_{L^2}^2+\kappa\|\nabla^3\tau\|_{L^2}^2 + \eta_1 \langle \Lambda^3 u^h,\Lambda^2\sigma^h \rangle,\,\,\,\,\,\\ 
                \mathcal{H}^h_3(t) :=\alpha\|\nabla^3u^h\|_{L^2}^2+\kappa\|\nabla^3\tau^h\|_{L^2}^2,~~\tilde{\mathcal{H}}^h_3(t):=\alpha\|\nabla^3u^h\|_{L^2}^2+\kappa\|\nabla^3\tau^h\|_{L^2}^2 + \eta_1 \langle \Lambda^3 u^h,\Lambda^2\sigma^h \rangle,
            \end{eqnarray*}  
            where $0<\eta_1\leq \frac{1}{2B_0}\min\{\alpha,\kappa\}$ is a constant defined in the proof of Lemma \ref{lemma_decay_u_3} to ensure that 
             $$\frac{1}{2}\mathcal{H}_3(t) \leq \tilde{\mathcal{H}}_3(t)  \leq 2\mathcal{H}_3(t),~  \text{ and } ~ \frac{1}{2}\mathcal{H}^h_3(t) \leq \tilde{\mathcal{H}}^h_3(t)  \leq 2\mathcal{H}^h_3(t).$$  The proof of Theorem \ref{thm_more_decay} consists of Lemmas \ref{lemma_decay_u_3}  and  \ref{lemma_tau_decay}. To begin with, we have following optimal decay estimate for $\|\nabla^3 u\|_{L^2}$ and $\|\nabla^2 \tau\|_{L^2}$ which is not obtained in \cite{HWWZ_2022}. A similar result of Lemma \ref{lemma_decay_u_3} for the compressible model was proved in \cite{Liu_Wang_Wen_2022}.

            \begin{lemma}\label{lemma_decay_u_3}
                Under the assumptions of Theorem \ref{thm_more_decay}, it holds that 
                \begin{eqnarray*}
                    \|\nabla^3 u  (t) \|_{L^2} + \|\nabla^2 \tau   (t) \|_{L^2} + \|\nabla^3 \tau  (t) \|_{L^2} \lesssim (1+t)^{-\frac{9}{4}}. 
                \end{eqnarray*}
            \end{lemma}
    \begin{proof}
                To begin with, applying $\nabla^3$ to (\ref{system}), and then multiplying (\ref{system})$_1$ by $\alpha\nabla^3 u$ and (\ref{system})$_2$ by $\kappa\nabla^3 \tau$, we have from integration by parts and the cancellation relation that
            \begin{equation} \label{eq_nabla_3}
                \begin{split}
                  &\frac12\frac{\mathrm{d}}{\mathrm{d}t}\left(\alpha\|\nabla^3u\|_{L^2}^2+\kappa\|\nabla^3\tau\|_{L^2}^2\right)+\alpha\epsilon\|\nabla^4 u\|_{L^2}^2+\kappa\mu\|\nabla^4\tau\|_{L^2}^2+\kappa\beta\|\nabla^3\tau\|_{L^2}^2
                 \\ =&
                 -\alpha \langle\nabla^3(u\cdot\nabla u),\nabla^3 u \rangle-\kappa \langle\nabla^3(u\cdot\nabla \tau),\nabla^3\tau \rangle
                 -\kappa\langle\nabla^2Q(\nabla u,\tau),\nabla^4\tau\rangle=:\sum_{p=1}^3\mathcal{I}_p.
            \end{split}
        \end{equation}
        Using the incompressible condition, { \eqref{eq_smallness},}   H\"older inequality,   Cauchy inequality,   Lemmas \ref{lemma_Sobolev} and  \ref{lemma_commutator},  we can estimate $\mathcal{I}_1,\mathcal{I}_2$ and $\mathcal{I}_3$ as follows.  
        For $\mathcal{I}_1$, we have  
        \begin{align*}
            \mathcal{I}_1 =& -\alpha \langle\nabla^3(u\cdot\nabla u),\nabla^3 u \rangle \\ 
             =& -\alpha  \langle\nabla^3(u\cdot\nabla u)
            -(u\cdot\nabla) \nabla^3 u ,\nabla^3 u \rangle \\ 
            \lesssim& \,\|\nabla^3(u\cdot\nabla u)-(u\cdot\nabla) \nabla^3 u \|_{L^2}\|\nabla^3 u\|_{L^2} \\ 
            \lesssim& \, \Big(\|\nabla u\|_{L^\infty}\|\nabla^3 u \|_{L^2}  + \|\nabla^3 u \|_{L^2}\|\nabla u\|_{L^\infty}\Big)\|\nabla^3 u \|_{L^2} \\ 
            \lesssim&\, \varepsilon_0 \|\nabla^3 u \|_{L^2}^2.
        \end{align*}
        Similarly, for $\mathcal{I}_2$, we have 
        \begin{align*}
            \mathcal{I}_2 =& -\kappa \langle\nabla^3(u\cdot\nabla \tau),\nabla^3\tau \rangle\\ 
            =& -\kappa \langle\nabla^3(u\cdot\nabla \tau)
            -(u\cdot\nabla) \nabla^3 \tau ,\nabla^3 \tau \rangle \\ 
            \lesssim& \,\|\nabla^3(u\cdot\nabla \tau)-(u\cdot\nabla) \nabla^3 \tau \|_{L^2}\|\nabla^3 \tau\|_{L^2}\\ 
            \lesssim& \,\Big(\|\nabla u\|_{L^\infty}\|\nabla^3 \tau \|_{L^2}+\|\nabla \tau\|_{L^\infty}\|\nabla^3 u \|_{L^2}\Big)\|\nabla^3 \tau \|_{L^2}\\ 
            \lesssim&\, \varepsilon_0 \big( \|\nabla^3 u \|_{L^2}^2 + \|\nabla^3 \tau \|_{L^2}^2\Big), 
        \end{align*}
        And, for $\mathcal{I}_3$, we have 
        \begin{align*}
            \mathcal{I}_3 =&\,-\kappa\langle\nabla^2Q(\nabla u,\tau),\nabla^4\tau\rangle\\ 
             \lesssim  &\, \langle | \nabla u ||\nabla^2\tau |,|\nabla^4\tau|\rangle + \langle | \nabla^2 u | |\nabla\tau |,|\nabla^4\tau|\rangle
            +\langle | \nabla^3 u | |\tau |,|\nabla^4\tau|\rangle \\ 
            \lesssim &\, \Big(\|\nabla u\|_{L^3}\|\nabla^2 \tau\|_{L^6} + \|\nabla^2 u\|_{L^6}\|\nabla  \tau\|_{L^3}+ \|\nabla^3 u\|_{L^2}\|  \tau\|_{L^\infty}\Big)\|\nabla^4 \tau \|_{L^2}\\ 
            \leq&\,\frac{\kappa\mu}{2}\|\nabla^4 \tau \|_{L^2}^2 + C\varepsilon_0^2 \Big( \|\nabla^3 u \|_{L^2}^2 + \|\nabla^3 \tau \|_{L^2}^2\Big)
        \end{align*}
        Combining the above three estimate with \eqref{eq_nabla_3}, we obtain 
        \begin{equation} \label{eq_nabla_3_dissipation_tau}
            \begin{split}
              & \frac{\mathrm{d}}{\mathrm{d}t}\left(\alpha\|\nabla^3u\|_{L^2}^2+\kappa\|\nabla^3\tau\|_{L^2}^2\right) +\kappa\mu\|\nabla^4\tau\|_{L^2}^2+\kappa\beta\|\nabla^3\tau\|_{L^2}^2 
             \leq  B_1\varepsilon_0  \Big( \|\nabla^3 u \|_{L^2}^2 + \|\nabla^3 \tau \|_{L^2}^2\Big).
        \end{split}
    \end{equation}
    Next, to supplement the dissipation of $u$,  multiplying  $\Lambda^3\phi_1(D)(\ref{u-sigma})_1$ and $\Lambda^2\phi_1(D)(\ref{u-sigma})_2$  by $\Lambda^2\sigma^h$ and $\Lambda^3u^h$,  respectively, then summing the result up, we have, from  integration by parts, that 
    \begin{eqnarray*}
        \begin{aligned}
            & \frac{\mathrm{d}}{\mathrm{d}t} \langle \Lambda^3 u^h,\Lambda^2\sigma^h \rangle + \frac{\alpha}{2} \|\Lambda^3 u^h\|_{L^2}^2 \\
            = &\,\kappa\|\Lambda^3 \sigma^h\|_{L^2}^2 - (\epsilon+\mu)\langle \Lambda^3 u^h,\Lambda^4 \sigma^h\rangle -\beta \langle \Lambda^3 u^h,\Lambda^2 \sigma^h\rangle
            +\langle \Lambda^2 \mathcal{F}_1^h,\Lambda^3 \sigma^h\rangle
            +\langle \Lambda^3 u^h,\Lambda^2 \mathcal{F}_2^h\rangle \\ 
            \lesssim &\,\Big(\|\Lambda^4 \sigma^h\|_{L^2} + \|\Lambda^3 u^h\|_{L^2}   \Big)\|\Lambda^4 \sigma^h\|_{L^2}
            +   |\langle \Lambda^2 \mathcal{F}_1^h,\Lambda^3 \sigma^h\rangle|+ |\langle \Lambda^3 u^h,\Lambda^2 \mathcal{F}_2^h\rangle|,\\  
        \end{aligned}
    \end{eqnarray*} 
    where the last two terms can be estimated as follows, 
    \begin{align*}
        |\langle \Lambda^2 \mathcal{F}_1^h,\Lambda^3 \sigma^h\rangle| \lesssim &\, \|\Lambda^4 \sigma^h\|_{L^2}
        \Bigl( 
                \|u\|_{L^\infty}\|\nabla^3 u\|_{L^2}
                +\|\nabla u\|_{L^3}\|\nabla^2 u \|_{L^6}
                +\|\nabla^2 u\|_{L^6}\|\nabla u\|_{L^3}
             \Bigr)\\ 
        \lesssim &\, \|\Lambda^4 \sigma^h\|_{L^2} \|\nabla u\|_{H^1}  \|\nabla^3 u\|_{L^2},
    \end{align*}
    and
    \begin{eqnarray*}
        \begin{aligned}
             |\langle \Lambda^3 u^h,\Lambda^2 \mathcal{F}_2^h\rangle|\lesssim & \,\|\Lambda^3 u^h\|_{L^2}\|\Lambda^2 \mathcal{F}_2^h\|_{L^2}\\ 
            \lesssim&\, \|\nabla^3 u \|_{L^2}\Bigl( 
                \|u\|_{L^\infty}\|\nabla^3 \tau\|_{L^2}
                +\|\nabla u\|_{L^3}\|\nabla^2 \tau \|_{L^6}
                +\|\nabla^2 u\|_{L^6}\|\nabla \tau\|_{L^3}
             \Bigr)\\ 
             &+ \|\nabla^3 u \|_{L^2}\Bigl( 
                \|\nabla u\|_{L^3}\|\nabla^2 \tau\|_{L^6}
                +\|\nabla^2 u\|_{L^6}\|\nabla \tau \|_{L^3}
                +\|\nabla^3 u\|_{L^2}\| \tau\|_{L^\infty}
             \Bigr)\\ 
             \lesssim &\, \|\nabla^3 u \|_{L^2} \Bigl(\|\nabla u\|_{H^1}\|\nabla^3 \tau \|_{L^2} + \|\nabla \tau \|_{H^1}\|\nabla^3 u \|_{L^2} \Bigr) \\ 
             \lesssim & \|\nabla(u,\tau)\|_{H^1}^2\Big( \|\nabla^3 u \|_{L^2}^2 + \|\nabla^3 \tau \|_{L^2}^2\Big).
        \end{aligned}
    \end{eqnarray*}
    Combining the above three inequalities, we   obtain that
    \begin{eqnarray*}
        \begin{aligned}
            & \frac{\mathrm{d}}{\mathrm{d}t} \langle \Lambda^3 u^h,\Lambda^2\sigma^h \rangle + \frac{\alpha}{2} \|\Lambda^3 u^h\|_{L^2}^2 \\
            \leq&\,\frac{\alpha}{4}\|\Lambda^3 u^h\|_{L^2}^2 +  C\|\Lambda^4 \tau^h\|_{L^2}^2 + C\|\nabla(u,\tau)\|_{H^1}^2\Big( \|\nabla^3 u \|_{L^2}^2 + \|\nabla^3 \tau \|_{L^2}^2\Big),
        \end{aligned}
    \end{eqnarray*} 
    which implies that 
    \begin{eqnarray}\label{eq_u_dissip}
          \frac{\mathrm{d}}{\mathrm{d}t} \langle \Lambda^3 u^h,\Lambda^2\sigma^h \rangle + \frac{\alpha}{4} \|\Lambda^3 u^h\|_{L^2}^2
          \leq B_2\|\Lambda^4 \tau^h\|_{L^2}^2 + B_2\varepsilon_0 \Big( \|\nabla^3 u \|_{L^2}^2 + \|\nabla^3 \tau \|_{L^2}^2\Big).
    \end{eqnarray}
    Set $\eta_1 = \min\left\{\frac{1}{2B_0}\min\{\alpha,\kappa\}, \frac{\kappa\mu}{2B_2}\right\}.$ Then, multiplying \eqref{eq_u_dissip} by $\eta_1$, summing the result with \eqref{eq_nabla_3_dissipation_tau},   and  assuming  
    \begin{eqnarray}\label{eq_varepsilon}
       \varepsilon_0\leq \min\left\{\frac{\alpha}{16B_2},\frac{\alpha \eta_1}{16B_1},\frac{\kappa\beta}{4B_2\eta_1}, \frac{\kappa\beta}{4B_1}\right\}, 
    \end{eqnarray}   we  obtain that 
    \begin{eqnarray*} 
        \frac{\mathrm{d}}{\mathrm{d}t}\tilde{\mathcal{H}}_3(t) 
        +\frac{\kappa\mu}{2}\|\nabla^4\tau\|_{L^2}^2+\frac{\kappa\beta}{2}\|\nabla^3\tau\|_{L^2}^2  
        +\frac{\eta_1\alpha}{8} \|\Lambda^3 u^h\|_{L^2}^2
        \leq (B_1+B_2)\varepsilon_0\|\nabla^3 u^\ell\|_{L^2}^2,
    \end{eqnarray*}
    which, together with the fact ``$\|\nabla^3 u\|_{L^2}^2 = \|\Lambda^3 u\|_{L^2}^2\leq \|\Lambda^3 u^\ell\|_{L^2}^2 + \|\Lambda^3 u^h\|_{L^2}^2$'', implies 
    \begin{eqnarray}\label{eq_nabla_3_u_tau_dissip}
        \frac{\mathrm{d}}{\mathrm{d}t}\tilde{\mathcal{H}}_3(t) 
        + \min\left\{\frac{\beta}{2},\frac{\eta_1}{8}\right\}\left(\kappa\|\nabla^3\tau\|_{L^2}^2  
        +\alpha \|\nabla^3 u\|_{L^2}^2\right)
        \leq \Big((B_1+B_2)\varepsilon_0 + \frac{\eta_1\alpha}{8}\Big)\|\nabla^3 u^\ell\|_{L^2}^2.
    \end{eqnarray}
    Next, letting $\eta_2:=\frac{1}{2}\min\left\{\frac{\beta}{2},\frac{\eta_1}{8}\right\}, B_3:= (B_1+B_2)\varepsilon_0 + \frac{\eta_1\alpha}{8}$, we have, from   \eqref{eq_nabla_3_u_tau_dissip}, that
    \begin{eqnarray}\label{eq_nabla_3_u_tau_dissip_f}
        \frac{\mathrm{d}}{\mathrm{d}t}\tilde{\mathcal{H}}_3(t) 
        +  \eta_2 \tilde{\mathcal{H}}_3(t)
        \leq B_3\|\nabla^3 u^\ell\|_{L^2}^2.
    \end{eqnarray}
    Using { H\"older's} inequality, {Sobolev's} inequality,  {  \eqref{eq_1}, \eqref{eq_2} } and Lemma \ref{low_feq_est}, we have that 
    \begin{align*}
        \left(
        \int_{|\xi|\leq R}
         |\xi|^{6}|\hat{u}|^2
          {\rm d}\xi
        \right)^\frac{1}{2}
         \lesssim &\,
         (1+t)^{-\frac{9}{4}}
           \|(u_0,\tau_0)\|_{L^1} 
           +
          \int_0^{\frac{t}{2}}
          (1+t-s)^{-\frac{9}{4}}
           \|(u,\tau)\|_{L^2} \|\nabla( u,\tau)\|_{L^2}\mathrm{d}s \\
           & 
          +
          \int_{\frac{t}{2}}^t
          (1+t-s)^{-\frac{3}{2}}
          \|(u,\tau)\|_{L^2} \|\nabla^2( u,\tau)\|_{H^1} \mathrm{d}s \\ 
          \lesssim&\,(1+t)^{-\frac{9}{4}} + \int_0^{\frac{t}{2}}
          (1+t-s)^{-\frac{9}{4}} (1+s)^{ -\frac{3}{4}-\frac{5}{4}}\mathrm{d}s \\
           & 
          +
          \int_{\frac{t}{2}}^t
          (1+t-s)^{-\frac{3}{2}} (1+s)^{ -\frac{3}{4}-\frac{7}{4}} \mathrm{d}s 
          \lesssim  
         (1+t)^{-\frac{9}{4}}.
    \end{align*}
    Then, using {Gronwall's} inequality and  \eqref{eq_nabla_3_u_tau_dissip_f}, we get
    \begin{eqnarray}\label{eq_nabla_3_u_decay}
        \begin{aligned}
            \mathcal{H}_3(t) \leq 
            2\tilde{\mathcal{H}}_3(t) \leq&\, \mathrm{e}^{-\eta_2t} \tilde{\mathcal{H}}_3(0) 
            +B_3\int_0^t \mathrm{e}^{- \eta_2 (t-s)} \left(
                \int_{|\xi|\leq R}
                 |\xi|^{6}|\hat{u}(s)|^2
                  {\rm d}\xi
                \right) \mathrm{d}s   
            \lesssim \,(1+t)^{-\frac{9}{2}}.
        \end{aligned}
    \end{eqnarray}
    Now, we are in the position to improve the decay estimate of $\|\nabla^2\tau\|_{L^2}.$ Applying $\nabla^2$ to \eqref{system}$_2$ multiplying the result by $\nabla^2\tau$, and integrating with respect to $x$, we have, from \eqref{eq_nabla_3_u_decay}, that 
    \begin{equation*} 
        \begin{split}
              \frac{1}{2}\frac{\mathrm{d}}{\mathrm{d}t} \|\nabla^2 \tau\|_{L^2}^2 + \mu \|\nabla^{3} \tau \|_{L^2}^2 + \frac{\beta}{2}\|\nabla^2 \tau\|_{L^2}^2 
             \lesssim &\, \alpha \|\nabla^{3}u\|_{L^2}^2 + \|\nabla^2{Q}(\nabla u,\tau)\|_{L^2}^2 + \|\nabla^{2}(u\cdot \nabla\tau)\|_{L^2}^2\\
             \lesssim &\, \alpha \|\nabla^{3}u\|_{L^2}^2 + \|\nabla^{3}u\|_{L^2}^2\|\tau\|_{L^\infty}^2 + \|\nabla u\|_{L^\infty}^2\|\nabla^{2}\tau\|_{L^2}^2\\
             &+ \|\nabla^{3}\tau\|_{L^2}^2\|u\|_{L^\infty}^2 + \|\nabla \tau\|_{L^\infty}^2\|\nabla^{2}u\|_{L^2}^2 \\ 
             \lesssim &\,  (1+ t)^{- \frac{9}{2}  },
        \end{split}
 \end{equation*}
 which, together with {Gronwall's} inequality, implies that 
 \begin{equation}\label{eq_tau_2}
    \|\nabla^2 \tau (t) \|_{L^2}^2 \lesssim e^{-\beta t}\|\nabla^2 \tau_0\|_{L^2}^2 + \int_{0}^{t} e^{-\beta(t-s)}(1+s)^{- \frac{9}{2}  }\mathrm{d} s \leq C (1+ t)^{- \frac{9}{2} }.
\end{equation}
The proof is complete by \eqref{eq_nabla_3_u_decay} and \eqref{eq_tau_2}. 
\end{proof}
Next, we prove the optimal decay estimate for $\|\nabla^3 \tau\|_{L^2}.$ 
    \begin{lemma}\label{lemma_tau_decay}
        Under the assumptions of Theorem \ref{thm_more_decay}, it holds that 
                \begin{eqnarray*}
                      \|\nabla^3 \tau  (t) \|_{L^2} \lesssim (1+t)^{-\frac{11}{4}}. 
                \end{eqnarray*}
    \end{lemma}
    \begin{proof}
        To begin with, multiplying  $\alpha\nabla^3\phi_1(D)(\ref{system})_1$ and $\kappa\nabla^3\phi_1(D)(\ref{system})_2$  by $\nabla^3 u^h$ and $\nabla^3 \tau^h$,  respectively, then summing the result up, we have, from  integration by parts, that 
        \begin{equation} \label{eq_nabla_3_h_freq}
            \begin{split}
              &\frac12\frac{\mathrm{d}}{\mathrm{d}t}\left(\alpha\|\nabla^3u^h\|_{L^2}^2+\kappa\|\nabla^3\tau^h\|_{L^2}^2\right)+\alpha\epsilon\|\nabla^4 u^h\|_{L^2}^2+\kappa\mu\|\nabla^4\tau^h\|_{L^2}^2+\kappa\beta\|\nabla^3\tau^h\|_{L^2}^2
             \\ =&
             -\alpha \langle\nabla^3(u\cdot\nabla u)^h,\nabla^3 u^h \rangle + \kappa \langle\nabla^2(u\cdot\nabla \tau)^h,\Delta\nabla^2\tau^h \rangle
             -\kappa\langle\nabla^2(Q(\nabla u,\tau))^h,\nabla^4\tau^h\rangle=:\sum_{p=1}^3\mathcal{J}_p.
        \end{split}
    \end{equation}
     Using  Parseval's relation, {the fact that $\tilde{\phi}_1(\xi) =  {\phi}_1^2(\xi)$,} we can rewrite $\mathcal{J}_1$ in the following form 
    \begin{align*}\mathcal{J}_1 =&\,-\alpha \langle\nabla^3(u\cdot\nabla u)^h,\nabla^3 u^h \rangle\\ 
        =&\,-\alpha \langle\mathscr{F} \big[\nabla^3(u\cdot\nabla u)^h\big], \mathscr{F} \big[\nabla^3 u^h \big]\rangle\\ 
     = &\,-\alpha \Big\langle (i\xi_j)(i\xi_m)(i\xi_n) \phi_1(\xi)\mathscr{F}[(u\cdot\nabla u)],(i\xi_j)(i\xi_m)(i\xi_n) \phi_1(\xi)\hat{u} \Big\rangle \\ 
     = &\,-\alpha \Big\langle (i\xi_j)(i\xi_m)(i\xi_n)  \mathscr{F}[(u\cdot\nabla u)],(i\xi_j)(i\xi_m)(i\xi_n) \phi_1^2(\xi)\hat{u} \Big\rangle \\ 
     =&\,-\alpha \langle\nabla^3(u\cdot\nabla u) ,\nabla^3 u^{\tilde{h}} \rangle.
    \end{align*} 
    Then, using the incompressible condition and decomposition \eqref{eq_decom}, for $\mathcal{J}_1$
    we have 
    \begin{align*}
        \mathcal{J}_1 =& -\alpha \langle\nabla^3(u\cdot\nabla u) ,\nabla^3 u^{\tilde{h}} \rangle \\ 
        =& -\alpha \langle\nabla^3(u\cdot\nabla u^{\tilde{\ell}}) ,\nabla^3 u^{\tilde{h}} \rangle 
        -\alpha \langle\nabla^3(u\cdot\nabla u^{\tilde{h}}) -(u\cdot\nabla) \nabla^3 u^{\tilde{h}},\nabla^3 u^{\tilde{h}} \rangle\\ 
        =:&\, \mathcal{J}_{11} +  \mathcal{J}_{12}.
    \end{align*}
    Using H\"older inequality,  Cauchy inequality,   { \eqref{eq_1}, \eqref{eq_2}}, Lemmas \ref{lemma_Sobolev}, \ref{lemma_commutator} and \ref{lemma_decay_u_3}, for $\mathcal{J}_{11}$, we have 
    \begin{align*}
        \mathcal{J}_{11} =&\, -\alpha \langle\nabla^3(u\cdot\nabla u^{\tilde{\ell}}) ,\nabla^3 u^{\tilde{h}} \rangle \\ 
        \lesssim &\, \Big(\|\nabla^3 u\|_{L^2}\|\nabla  u^{\tilde{\ell}}\|_{L^\infty} +\|u\|_{L^\infty}\|\nabla^4 u^{\tilde{\ell}}\|_{L^2}\Big)\| \nabla^3 u^{\tilde{h}}\|_{L^2} \\ 
        \lesssim &\, \Big(\|\nabla^3 u\|_{L^2}\|\nabla^2  u \|_{H^1} +\|\nabla u\|_{H^1}\|\nabla^4 u^{\tilde{\ell}}\|_{L^2}\Big)\| \nabla^3 u  \|_{L^2} \\ 
        \lesssim &\, \Big((1+t)^{- \frac{9}{4}-\frac{7}{4} }
        +(1+t)^{-\frac{5}{4} }\|\nabla^4 u^{\tilde{\ell}}\|_{L^2}
        \Big)(1+t)^{-\frac{9}{4} }\\ 
        \lesssim&\, (1+t)^{-\frac{25}{4} } + (1+t)^{-\frac{7}{2}}\|\nabla^4 u^{\tilde{\ell}}\|_{L^2}.
    \end{align*}
    Similarly, for $\mathcal{J}_{12}$, we have 
    \begin{align*}
        \mathcal{J}_{12} =&\, -\alpha \langle\nabla^3(u\cdot\nabla u^{\tilde{h}}) -(u\cdot\nabla) \nabla^3 u^{\tilde{h}},\nabla^3 u^{\tilde{h}} \rangle\\ 
        \lesssim&\,\Big( \|\nabla^3 u\|_{L^2}\|\nabla  u^{\tilde{h}}\|_{L^\infty}
        +\|\nabla u\|_{L^\infty}\|\nabla^3 u^{\tilde{h}}\|_{L^2}\Big)\| \nabla^3 u^{\tilde{h}}\|_{L^2}\\
        \lesssim&\,   \|\nabla^2  u \|_{H^1}
        \| \nabla^3 u \|_{L^2}^2\\  
        \lesssim& \,(1+t)^{-\frac{7}{4}-\frac{9}{4}\times 2} = (1+t)^{-\frac{25}{4}}.
    \end{align*}
    Combining the above two estimates, we have 
    \begin{equation}\label{eq_J_1}
        \mathcal{J}_{1} = \mathcal{J}_{11} + \mathcal{J}_{12}
        \lesssim  (1+t)^{-\frac{25}{4} } + (1+t)^{-\frac{7}{2}}\|\nabla^4 u^{\tilde{\ell}}\|_{L^2}.
    \end{equation}
    Next, using the incompressible condition,   H\"older inequality,  Cauchy inequality, the decomposition \eqref{eq_decom}, { \eqref{eq_1}, \eqref{eq_2}}, Lemmas \ref{lemma_Sobolev}, \ref{lemma_commutator} and \ref{lemma_decay_u_3}, we can estimate $\mathcal{J}_2$ and $\mathcal{J}_3$ as follows.
    \begin{equation} \label{eq_J_2}
        \begin{split} 
            \mathcal{J}_2 =&\,\kappa \langle\nabla^2(u\cdot\nabla \tau)^h,\Delta\nabla^2\tau^h \rangle 
            \lesssim  \langle |\nabla^2(u\cdot\nabla \tau)| , |\nabla^4\tau^h| \rangle\\
             \lesssim&  \Big(\|\nabla^2 u\|_{L^2}\|\nabla  \tau\|_{L^\infty} +\|u\|_{L^\infty}\|\nabla^3 \tau\|_{L^2}\Big)
            \|\nabla^4\tau^h\|_{L^2}\\ 
            \leq &\,\frac{\kappa\mu}{4}\|\nabla^4\tau^h\|_{L^2}^2
            + C\Big(\|\nabla^2 u\|_{L^2}^2\|\nabla^2 \tau\|_{H^1}^2 +\|\nabla u\|_{H^1}^2\|\nabla^3 \tau\|_{L^2}^2\Big) \\ 
            \leq&\, \frac{\kappa\mu}{4}\|\nabla^4\tau^h\|_{L^2}^2 + 
            C (1+t)^{-\frac{7}{2}-\frac{9}{2} } +  C (1+t)^{-\frac{5}{2}-\frac{9}{2} }  \\ 
            \leq&\, \frac{\kappa\mu}{4}\|\nabla^4\tau^h\|_{L^2}^2 + C(1+t)^{- 7 },
        \end{split}
    \end{equation}
        and 
        \begin{equation} \label{eq_J_3}
            \begin{split}
            \mathcal{J}_3 =& -\kappa\langle\nabla^2(Q(\nabla u,\tau))^h,\nabla^4\tau^h\rangle \\ 
            \lesssim &\,
            \Big(\|\nabla^3 u\|_{L^2}\|\tau\|_{L^\infty}
            +\|\nabla u\|_{L^\infty}\|\nabla^2 \tau\|_{L^2} \Big)\|\nabla^4\tau^h\|_{L^2} \\ 
            \leq&\, \frac{\kappa\mu}{4}\|\nabla^4\tau^h\|_{L^2}^2 
            +C\Big(\|\nabla^3 u\|_{L^2}^2\|\nabla\tau\|_{H^1}^2
            +\|\nabla^2 u\|_{H^1}^2\|\nabla^2 \tau\|_{L^2}^2 \Big) \\
            \leq&\,\frac{\kappa\mu}{4}\|\nabla^4\tau^h\|_{L^2}^2 + C(1+t)^{-\frac{9}{2} -\frac{7}{2} } + C(1+t)^{-\frac{7}{2} -\frac{9}{2} }\\  
            \leq&\,\frac{\kappa\mu}{4}\|\nabla^4\tau^h\|_{L^2}^2 + C(1+t)^{- 8 }.
        \end{split}
    \end{equation} 
   Substituting    \eqref{eq_J_1}, \eqref{eq_J_2} and \eqref{eq_J_3}   into \eqref{eq_nabla_3_h_freq}, we obtain that 
    \begin{eqnarray}\label{eq_nabla_3_h_1}
        \begin{aligned} 
            & \frac{\mathrm{d}}{\mathrm{d}t}\left(\alpha\|\nabla^3u^h\|_{L^2}^2+\kappa\|\nabla^3\tau^h\|_{L^2}^2\right) 
        + \kappa\mu\|\nabla^4\tau^h\|_{L^2}^2
         +\kappa\beta\|\nabla^3\tau^h\|_{L^2}^2 \\
         \lesssim&\, (1+t)^{-\frac{25}{4}} + (1+t)^{-\frac{7}{2}}\|\nabla^4 u^{\tilde{\ell}}\|_{L^2}.
        \end{aligned}
    \end{eqnarray}
    Using {H\"older's} inequality, {Sobolev's} inequality, { \eqref{eq_1}, \eqref{eq_2} } and  Lemmas \ref{low_feq_est}, \ref{lemma_Sobolev} and \ref{lemma_decay_u_3},  we have that 
    \begin{eqnarray}\label{eq_nabla_3_h_2}
        \begin{aligned} 
        \left(
        \int_{|\xi|\leq R}
         |\xi|^{8}|\hat{u}|^2
          {\rm d}\xi
        \right)^\frac{1}{2}
         \lesssim &\,
         (1+t)^{-\frac{11}{4}}
           \|(u_0,\tau_0)\|_{L^1} 
           +
          \int_0^{\frac{t}{2}}
          (1+t-s)^{-\frac{11}{4}}
           \|(u,\tau)\|_{L^2} \|\nabla( u,\tau)\|_{L^2}\mathrm{d}s \\
           & 
          +
          \int_{\frac{t}{2}}^t
          (1+t-s)^{- {2}}
          \|(u,\tau)\|_{L^2} \|\nabla^2( u,\tau)\|_{L^2}^{\frac{1}{2}} \|\nabla^3( u,\tau)\|_{L^2}^{\frac{1}{2}} \mathrm{d}s   \\ 
          \lesssim &\,  (1+t)^{-\frac{11}{4}} 
          +\int_0^{\frac{t}{2}}
          (1+t-s)^{-\frac{11}{4}}
            (1+s)^{-\frac{3}{4}}(1+s)^{-\frac{5}{4}} \mathrm{d}s
              \\ 
            &   +
          \int_{\frac{t}{2}}^t
          (1+t-s)^{- {2}} (1+s)^{-\frac{3}{4}}(1+s)^{-\frac{7}{4}\times \frac{1}{2}}(1+s)^{-\frac{9}{4}\times \frac{1}{2}} \mathrm{d}s    \\
          \lesssim&\,(1+t)^{-\frac{11}{4}}.
        \end{aligned}
    \end{eqnarray}
     Combining \eqref{eq_nabla_3_h_1} and \eqref{eq_nabla_3_h_2}, we get that
    \begin{eqnarray}\label{eq_nabla_3_h_3}
        \frac{\mathrm{d}}{\mathrm{d}t}\left(\alpha\|\nabla^3u^h\|_{L^2}^2+\kappa\|\nabla^3\tau^h\|_{L^2}^2\right)
        +\kappa\mu\|\nabla^4\tau^h\|_{L^2}^2 
        +\kappa\beta\|\nabla^3\tau^h\|_{L^2}^2
        \lesssim (1+t)^{-\frac{25}{4}}.
   \end{eqnarray}
   From the proof of \eqref{eq_u_dissip}, it is easy to deduce that
   \begin{eqnarray}\label{eq_u_dissip_plus}
    \begin{aligned}
      \frac{\mathrm{d}}{\mathrm{d}t} \langle \Lambda^3 u^h,\Lambda^2\sigma^h \rangle + \frac{\alpha}{4} \|\Lambda^3 u^h\|_{L^2}^2
    \leq &\, B_2\|\Lambda^4 \tau^h\|_{L^2}^2 +\|\nabla(u,\tau)\|_{H^1}^2\Big( \|\nabla^3 u \|_{L^2}^2 + \|\nabla^3 \tau \|_{L^2}^2\Big)  \\ 
    \leq &\,  B_2\|\Lambda^4 \tau^h\|_{L^2}^2 + C(1+t)^{-7}.
    \end{aligned} 
\end{eqnarray} 
 Then, \eqref{eq_nabla_3_h_3} along with $ \eta_1$\eqref{eq_u_dissip_plus} imply that  
\begin{eqnarray}\label{eq_nabla_3_u_tau_dissip_f_h_freq}
    \frac{\mathrm{d}}{\mathrm{d}t}\tilde{\mathcal{H}}^h_3(t) 
    +  \eta_2 \tilde{\mathcal{H}}^h_3(t)
    \lesssim  (1+t)^{-\frac{25}{4}}.
\end{eqnarray}
Using {Gronwall's} inequality, we get, from  \eqref{eq_nabla_3_u_tau_dissip_f_h_freq}, that 
\begin{eqnarray}\label{eq_nabla_3_tau_h_decay}
    \|\nabla^3 (u^h,\tau^h) (t) \|_{L^2}^2\lesssim {\mathcal{H}}^h_3(t)\lesssim
    \mathrm{e}^{-\eta_2t} \tilde{\mathcal{H}}_3^h(0) 
            + \int_0^t \mathrm{e}^{- \eta_2 (t-s)} (1+s)^{-\frac{25}{4}} \mathrm{d}s   
                 \lesssim (1+t)^{-\frac{25}{4}}.
\end{eqnarray}
Applying $\nabla^3$ to \eqref{system}$_2$ multiplying the result by $\nabla^3\tau$, and integrating with respect to $x$,
 we get, 
\begin{equation}\label{eq_tau_last}
    \begin{split}
        &\frac{1}{2}\frac{\mathrm{d}}{\mathrm{d}t} \|\nabla^3 \tau\|_{L^2}^2 + \mu \|\nabla^{4} \tau \|_{L^2}^2 +  {\beta} \|\nabla^3 \tau\|_{L^2}^2 \\ 
           = &\, \langle\nabla^2(u\cdot\nabla \tau) ,\Delta\nabla^2\tau  \rangle
           -\langle\nabla^2(Q(\nabla u,\tau)) ,\nabla^4\tau  \rangle
           +\langle \nabla^3\mathbb{D} u^{\ell} ,\nabla^3 \tau\rangle
           -\langle \nabla^3  u^{h} ,\nabla^3 \mathrm{div}\tau\rangle   
           =:  \sum_{j = 1}^{4}\mathcal{K}_{j}.
    \end{split}
\end{equation}
Using { \eqref{eq_1}, \eqref{eq_2}}, Lemmas \ref{lemma_Sobolev}
 and \ref{lemma_decay_u_3}, we have, for $\mathcal{K}_{1}$, that 
\begin{align*}
    \mathcal{K}_{1} =&\, \langle\nabla^2(u\cdot\nabla \tau) ,\Delta\nabla^2\tau  \rangle \\ 
    \lesssim &\, \Bigl(\|u\|_{L^\infty}\|\nabla^3\tau\|_{L^2} + \|\nabla^2 u\|_{L^2}\|\nabla\tau\|_{L^\infty}\Bigr)\|\nabla^4 \tau\|_{L^2}^2\\ 
    \leq &\,\frac{\mu}{6}\|\nabla^4 \tau\|_{L^2}^2 + 
    C\Bigl(\|\nabla u\|_{H^1}^2\|\nabla^3\tau\|_{L^2}^2 + \|\nabla^2 u\|_{L^2}^2\|\nabla^2\tau\|_{H^1}^2\Bigr) \\ 
    \leq&\, \frac{\mu}{6}\|\nabla^4 \tau\|_{L^2}^2 
    + C (1+t)^{-\frac{5}{2}-\frac{9}{2}} + C (1+t)^{-\frac{7}{2}-\frac{9}{2}} \\ 
    \leq&\, \frac{\mu}{6}\|\nabla^4 \tau\|_{L^2}^2 + C (1+t)^{ - 7}.
\end{align*}
Similarly, for $\mathcal{K}_{2}$, we have 
\begin{align*}
    \mathcal{K}_{2} =&\,-\langle\nabla^2(Q(\nabla u,\tau)) ,\nabla^4\tau  \rangle \\ 
    \lesssim&\,   \Bigl(\|\nabla u\|_{L^\infty}\|\nabla^2\tau\|_{L^2} + \|\nabla^3 u\|_{L^2}\| \tau\|_{L^\infty}\Bigr)\|\nabla^4 \tau\|_{L^2}^2\\ 
    \leq &\,\frac{\mu}{6}\|\nabla^4 \tau\|_{L^2}^2 + 
    C\Bigl(\|\nabla^2 u\|_{H^1}^2\|\nabla^2\tau\|_{L^2}^2 + \|\nabla^3 u\|_{L^2}^2\|\nabla \tau\|_{H^1}^2\Bigr) \\ 
    \leq&\, \frac{\mu}{6}\|\nabla^4 \tau\|_{L^2}^2 
    + C (1+t)^{-\frac{7}{2}-\frac{9}{2}} + C (1+t)^{-\frac{9}{2}-\frac{7}{2}} \\ 
    \leq&\, \frac{\mu}{6}\|\nabla^4 \tau\|_{L^2}^2 + C (1+t)^{ - 8}.
\end{align*}
For  $\mathcal{K}_{3}$, we have 
{
\begin{align*}
    \mathcal{K}_{3} =&\, \langle \nabla^3\mathbb{D} u^{\ell} ,\nabla^3 \tau\rangle  \\  
    \lesssim &\,   \|\nabla^4  u^{\ell}\|_{L^2} \|\nabla^3 \tau\|_{L^2}\\ 
    \leq &\,\frac{\beta}{2}\|\nabla^3 \tau\|_{L^2} + C\|\nabla^4  u^{\ell}\|_{L^2} ^2 \\ 
    \leq&\,\frac{\beta}{2}\|\nabla^3 \tau\|_{L^2} 
    +(1+t)^{-\frac{11}{2}},
\end{align*}
}
where  \eqref{eq_nabla_3_h_2} is used. 
Moreover, using \eqref{eq_nabla_3_tau_h_decay},  for $\mathcal{K}_{4}$, we have
{
\begin{align*}
    \mathcal{K}_{4} = &\, -\langle \nabla^3  u^{h} ,\nabla^3 \mathrm{div}\tau\rangle  \\ 
    \lesssim &\,  \|\nabla^3 u^{h}\|_{L^2} \|\nabla^4 \tau\|_{L^2}\\ 
    \leq &\, \frac{\mu}{6}\|\nabla^4 \tau\|_{L^2}^2 + C\|\nabla^3 u^{h}\|_{L^2} \\ 
    \leq &\, \frac{\mu}{6}\|\nabla^4 \tau\|_{L^2}^2 + C(1+t)^{-\frac{25}{4}}.
\end{align*}
}
Substituting the above four estimates into \eqref{eq_tau_last}, we get  
\begin{eqnarray}\label{eq_tau_3_decay_1}
     \frac{\mathrm{d}}{\mathrm{d}t} \|\nabla^3 \tau\|_{L^2}^2 + \mu \|\nabla^{4} \tau \|_{L^2}^2 +  {\beta} \|\nabla^3 \tau\|_{L^2}^2 
     \lesssim (1+ t)^{- \frac{11}{2}  }.
\end{eqnarray}
Using {Gronwall's} inequality, we get, from  \eqref{eq_tau_3_decay_1}, that 
\begin{eqnarray*}\label{eq_nabla_3_tau_decay_final}
    \|\nabla^3 \tau (t)  \|_{L^2}^2 \lesssim
    \mathrm{e}^{-\beta t} \|\nabla^3 \tau_0 \|_{L^2}^2
            + \int_0^t \mathrm{e}^{- \beta (t-s)} (1+s)^{-\frac{11}{2}} \mathrm{d}s   
                 \lesssim (1+t)^{-\frac{11}{2}}.
\end{eqnarray*}
The proof is complete.
    \end{proof}
    {
    \begin{remark}
        The key observation in the proof of Lemma \ref{lemma_tau_decay} is that the high frequency part $(u^{\tilde{h}},\tau^{\tilde{h}})$ decays faster and the low frequency part $(u^{\tilde{\ell}},\tau^{\tilde{\ell}})$ enjoys better regularity.
    \end{remark}
    }
    \section*{Acknowledgement}
    The author is grateful to Professor Huanyao Wen for the helpful discussions.  The author would like to thank the anonymous referees for the valuable comments and suggestions. 

    \end{document}